\documentclass[reqno]{amsart}
\usepackage{amsthm,amsmath,amssymb}
\usepackage{stmaryrd,mathrsfs}
\usepackage[T1]{fontenc}
\usepackage{esint}
\usepackage{tikz}

\allowdisplaybreaks

\newcommand{\abs}[1]{|#1|}
\newcommand{\Babs}[1]{\Big|#1\Big|}

\newcommand{\Norm}[2]{\|#1\|_{#2}}

\newcommand{\BNorm}[2]{\Big\|#1\Big\|_{#2}}
\newcommand{\pair}[2]{\langle #1,#2 \rangle}
\newcommand{\Bpair}[2]{\Big\langle #1,#2 \Big\rangle}

\newcommand{\Cave}[1]{\langle\!\langle #1\rangle\!\rangle}

\newcommand{\lspan}[0]{\operatorname{span}}

\newcommand{\kernel}[0]{\mathsf{N}}
\newcommand{\range}[0]{\mathsf{R}}

\newcommand{\R}{\mathbb{R}}

\newcommand{\N}{\mathbb{N}}

\newcommand{\eps}[0]{\varepsilon}

\swapnumbers \numberwithin{equation}{section}

\theoremstyle{plain}
\newtheorem{theorem}[equation]{Theorem}
\newtheorem{proposition}[equation]{Proposition}
\newtheorem{corollary}[equation]{Corollary}
\newtheorem{lemma}[equation]{Lemma}

\theoremstyle{definition}
\newtheorem{definition}[equation]{Definition}

\theoremstyle{remark}
\newtheorem{remark}[equation]{Remark}
\newtheorem{example}[equation]{Example}

\makeatletter
\@namedef{subjclassname@2010}{%
  \textup{2010} Mathematics Subject Classification}
\makeatother

\begin{document}

\title{Reduced inequalities for vector-valued functions}

\begin{abstract}
Building on the notion of convex body domination introduced by Nazarov, Petermichl, Treil, and Volberg, we provide a general principle of bootstrapping bilinear estimates for scalar-valued functions into vector-valued versions with a reduced right-hand side involving iterated norms of a pointwise dot product $\vec f(x)\cdot\vec g(y)$ instead of the product of lengths $\abs{\vec f(x)}\abs{\vec g(y)}$ that would result from a na\"ive extension of the scalar inequality. On the way, we study connections between convex body domination and tensor norms. In order to cover the full regime of $L^p$ norms, also with $p<1$, that naturally arise in bilinear harmonic analysis, we develop a framework in general quasi-normed spaces. A key application is a vector-valued Kato--Ponce inequality (or fractional Leibnitz rule) with a reduced right-hand side, which we obtain as a soft corollary of the known scalar-valued version and our general bootstrapping method.
\end{abstract}

\author[T.~P.\ Hyt\"onen]{Tuomas P.\ Hyt\"onen}
\address{Department of Mathematics and Systems Analysis, Aalto University, P.O. Box~11100 (Otakaari~1), FI-00076 Aalto, Finland}
\email{tuomas.p.hytonen@aalto.fi}

\date{}%\today}

\thanks{The author was supported by the Research Council of Finland (project 346314: Finnish Centre of Excellence in Randomness and Structures, ``FiRST'')}

\keywords{Bilinear operator, convex body domination, quasi-norm, tensor product, vector-valued function}
\subjclass[2010]{15A69, 26D15, 47G10}
% 15A69 Multilinear algebra, tensor calculus
% 26D15 Inequalities for sums, series and integrals
% 47G10 Integral operators

% 42B20 Singular and oscillatory integrals (Calder\'on-Zygmund, etc.) 
% 42B25 Maximal functions, Littlewood-Paley theory
% 42B35 Function spaces arising in harmonic analysis 
% 42B37 Harmonic analysis and PDEs [See also 35-XX]

% 46B09 Probabilistic methods in Banach space theory
% 46E40 Spaces of vector- and operator-valued functions
% 47A60 Functional calculus
% 47F05 Partial differential operators
% 60G46 Martingales and classical analysis

\maketitle

\section{Introduction}\label{sec:intro}

As a tool for proving weighted norm inequalities with matrix weights, Nazarov, Petermichl, Treil, and Volberg \cite{NPTV:convex} introduced the notion of {\em convex body domination}, which elaborates on Lerner's \cite{Lerner:simple} celebrated idea of {\em sparse domination} by capturing more precise directional information of vector-valued functions. For the action of the Hilbert transform and other Calder\'on--Zygmund operators on such functions, this method allowed Nazarov et al.~\cite{NPTV:convex} to prove the best available matrix-weighted estimate, which was subsequently shown to be sharp by Domelevo, Petermichl, Treil, and Volberg \cite{DPTV}. %Besides the original application, convex body domination has also found applications to commutator

Sparse domination of a bilinear form $\tau:F\times G\to\R$ refers to majorisation of the form
\begin{equation}\label{eq:sparseDom}
  \abs{\tau(f,g)}\leq C\sum_{Q\in\mathscr S}\Norm{f}{X(Q)}\Norm{g}{Y(Q)},
\end{equation}
where $\mathscr S$ is a {\em sparse} collection of dyadic cubes (i.e., each $Q\in\mathscr S$ has a {\em major subset} $E_Q\subseteq Q$ with $\abs{E_Q}\geq\gamma\abs{Q}$, and these subsets $E_Q$ are pairwise disjoint), and $X(Q)$ and $Y(Q)$ are normed spaces indexed by all dyadic cubes $Q$. Here the collection $\mathscr S$, but not the quantitative parameters $\gamma$ and $C$, is allowed to depend on the functions $f$ and $g$.
Following a recent abstraction by the author~\cite{Hyt:convex}, convex body domination of \cite{NPTV:convex} is based on the idea that, when the bilinear form is extended to act on vectors $\vec f\in F^n$ and $\vec g\in G^n$ via
\begin{equation*}
  \tau(\vec f,\vec g):=\sum_{i=1}^n\tau(f_i,g_i),
\end{equation*}
one should aim for replacing each product of norms
\begin{equation*}
  \Norm{f}{X}\Norm{g}{Y}
  =\sup\{\pair{f}{x^*}\pair{g}{y^*}:x^*\in\bar B_{X^*},y^*\in\bar B_{Y^*}\}
\end{equation*}
by the quantity
\begin{equation}\label{eq:theQuantity}
  \sup\Big\{\sum_{i=1}^n\pair{f_i}{x^*}\pair{g_i}{y^*}:x^*\in\bar B_{X^*},y^*\in\bar B_{Y^*}\Big\},
\end{equation}
where $\bar B_{Z^*}$ is the closed unit ball of the normed dual $Z^*$ of $Z\in\{X,Y\}$.
In \cite{Hyt:convex}, the set inside the supremum is interpreted as the Minkowski dot product $A\cdot B=\{\vec a\cdot\vec b:\vec a\in A,\vec b\in B\}$ of the convex bodies (hence the name!)
\begin{equation*}
  \Cave{\vec f}_X:=\big\{(\pair{f_i}{x^*})_{i=1}^n:x^*\in\bar B_{X^*}\big\},\quad
  \Cave{\vec g}_Y:=\big\{(\pair{g_i}{x^*})_{i=1}^n:y^*\in\bar B_{Y^*}\big\}\subseteq\R^n;
\end{equation*}
then $\Cave{\vec f}_X\cdot\Cave{\vec g}_Y$ is a symmetric interval of the form $[-c,c]\subseteq\R$ that one can identify with its right end-point $c$, which is the supremum in \eqref{eq:theQuantity}. This is indeed an abstraction of the earlier studies that dealt with the case $X=Y=L^1(Q)$ \cite{CDPO:unif,NPTV:convex} and $X,Y\in\{L^p(Q):p\in[1,\infty)\}$ \cite{DPHL,MRR:22}.
However, the reader may also recognise \eqref{eq:theQuantity} as the injective norm $\Norm{\vec f\odot\vec g}{X\otimes_\eps Y}$ of the ``tensor dot product''
\begin{equation*}
  \vec f\odot\vec g:=\sum_{i=1}^n f_i\otimes g_i\in X\otimes Y;
\end{equation*}
see e.g. \cite[Chapter 3]{Ryan:book}.

It is an outstanding open problem whether sparse domination of a given operator always implies convex body domination. Starting with \cite{NPTV:convex}, this has been verified in several concrete cases \cite{CDPO:unif,DPHL,Hyt:convex,Lauk:BRS,MRR:22} by revisiting the proof of the respective sparse domination result. Very roughly speaking, and certainly cutting some corners, these proofs work as follows:
\begin{enumerate}
  \item Behind each proof of \eqref{eq:sparseDom}, one identifies an implicit decomposition
\begin{equation*}
  \tau(f,g)=\sum_{Q\in\mathscr S}\tau_Q(f,g),
\end{equation*}
where each $\tau_Q$ is a bilinear form with the simple estimate
\begin{equation}\label{eq:basicEst}
  \abs{\tau_Q(f,g)}\lesssim\Norm{f}{X(Q)}\Norm{g}{X(Q)}.
\end{equation}
  \item An estimate of the form \eqref{eq:basicEst} bootstraps to the vector-valued estimate
\begin{equation}\label{eq:convexEst}
  \abs{\tau_Q(\vec f,\vec g)}\lesssim\Cave{\vec f}_{X(Q)}\cdot\Cave{\vec g}_{Y(Q)}=\Norm{\vec f\odot\vec g}{X(Q)\otimes_\eps Y(Q)};
\end{equation}
this is \cite[Proposition 4.2]{Hyt:convex}, where the implied constant may depend on the length $n$ of the vectors.
\end{enumerate}

While both upper bounds in \eqref{eq:convexEst} may seem a bit exotic, it is useful to observe the important special case
\begin{equation}\label{eq:iterLpLq}
  \Norm{\vec f\odot\vec g}{L^p\otimes_\eps L^q}\leq\min\Big\{ \Norm{\vec f\odot\vec g}{L^p(L^q)}, \Norm{\vec f\odot\vec g}{L^q(L^p)}\Big\},
\end{equation}
where on the right we have the iterated $L^p(L^q)$ (or $L^q(L^p)$) norm of
\begin{equation*}
  (\vec f\odot\vec g)(x,y)=\vec f(x)\cdot\vec g(y);
\end{equation*}
thus the upper bound in \eqref{eq:iterLpLq} is smaller than $\Norm{\vec f}{L^p}\Norm{\vec g}{L^q}$ by the amount that the dot product $\vec f(x)\cdot\vec g(y)$ is smaller than the product of the lengths $\abs{\vec f(x)}\abs{\vec g(y)}$ when integrated over all $(x,y)$. In fact, for any fixed dimension $n$, the two sides of \eqref{eq:iterLpLq} are comparable; see Propositions \ref{prop:equiNormsBanach} and \ref{prop:equiNormsQBanachFn}.

This observation shows that, aside from the previously explored applications to convex body domination, the basic implication \eqref{eq:basicEst} $\Rightarrow$  \eqref{eq:convexEst} of \cite[Proposition 4.2]{Hyt:convex} already allows the bootstrapping of various bilinear $L^p$-estimates of classical analysis. For example, with $\tau(f,g)=\int fg$, we find that the usual H\"older's inequality
\begin{equation*}
  \Babs{\int fg}\leq\Norm{f}{L^p}\Norm{g}{L^{q}},\quad \frac1p+\frac1q=1,
\end{equation*}
implies the vector-valued improvement
\begin{equation}\label{eq:betterHolder}
  \Babs{\int\vec f\cdot\vec g}
  \lesssim\begin{cases} \Norm{(x,y)\mapsto \vec f(x)\cdot\vec g(y) }{L^p_x L^q_y }, \\ 
   \Norm{(x,y)\mapsto \vec f(x)\cdot\vec g(y) }{L^q_y L^p_x },\end{cases}
\end{equation}
where the implied constant depends only on the dimension $n$ of the vectors. On the other hand, such a dependence is easily seen to be necessary:

\begin{example}\label{ex:needDim}
Let $\vec f=\vec g=\sum_{i=1}^n 1_{[i-1,i)}\vec e_i$, where $\vec e_i$ are orthonormal. Then
\begin{equation*}
  \vec f(x)\cdot\vec g(y)=\sum_{i,j=1}^n 1_{[i-1,i)}(x)1_{[j-1,j)}(y)\vec e_i\cdot\vec e_j
  =\sum_{i=1}^n 1_{[i-1,i)}(x)1_{[i-1,i)}(y).
\end{equation*}
Thus
\begin{equation*}
  \int \vec f\cdot\vec g=\int 1_{[0,n)}=n
\end{equation*}
while
\begin{equation*}
\begin{split}
  &\Norm{(x,y)\mapsto \vec f(x)\cdot\vec g(y)}{L^p_xL^q_y}=\Norm{1_{[0,n)}}{L^p}=n^{1/p}, \\
  &\Norm{(x,y)\mapsto \vec f(x)\cdot\vec g(y)}{L^q_yL^p_x}=\Norm{1_{[0,n)}}{L^q}=n^{1/q},
\end{split}
\end{equation*}
which are both $o(n)$ as $n\to\infty$ for each $p\in(1,\infty)$.
\end{example}

The aim of this paper is to place this vector-valued bootstrapping phenomenon into a more general framework so as to cover a larger class of non-trivial applications. In particular, in place of mere bilinear forms (mappings into $\R$), we wish to cover bilinear operators, with a function space target such as $L^p$. For such operators, there are multiple interesting examples in harmonic analysis, where the target space exponent may be in the range $p\in(0,1)$. While most of the interesting bilinear operators still have domain spaces in the usual regime $p>1$, the basic question of a vector-valued version
\begin{equation}\label{eq:HolderVec}
  \Norm{x\mapsto\vec f(x)\cdot\vec g(x)}{L^r_x}\overset{?}{\lesssim}
  \Norm{(x,y)\mapsto \vec f(x)\cdot\vec g(y)}{L^p_x(L^q_y)}
\end{equation}
 of the general H\"older's inequality
\begin{equation}\label{eq:HolderSca}
  \Norm{fg}{r}\leq\Norm{f}{p}\Norm{g}{q},\quad \frac1p+\frac1q=\frac1r,\quad p,q,r\in(0,\infty],
\end{equation}
where all exponents may be smaller than $1$, suggests that the result should find its natural generality entirely within the setting of quasi-normed spaces. Besides a constant in the triangle inequality, we should recall that the $L^p$ spaces with $p\in(0,1)$ suffer from a total lack of nontrivial bounded linear functionals, i.e., the dual space is $(L^p)^*=\{0\}$. This makes the quantity \eqref{eq:theQuantity} completely useless, and requires us to seek for an alternative framework. 

Yet another level of desirable generality comes from examples in the style of the Kato--Ponce inequality (or fractional Leibnitz rule), in its  original form from \cite{KatoPonce}, but we state here a version obtained by Grafakos and Oh \cite{GraOh:14}:

\begin{theorem}[\cite{GraOh:14}, Theorem 1]\label{thm:GraOh}
Let $p_k,q_k\in(1,\infty]$ and $r\in(\frac12,\infty]$ satisfy $\frac1r=\frac{1}{p_k}+\frac{1}{q_k}$ for $k=0,1$, and let $s>d(\frac1r-1)_+$ or $s\in 2\N$. Then
\begin{equation}\label{eq:GraOh}
  \Norm{D^s(fg)}{L^r(\R^d)}
  \lesssim\Norm{D^s f}{L^{p_0}(\R^d)}\Norm{g}{L^{q_0}(\R^d)}+\Norm{f}{L^{p_1}(\R^d)}\Norm{D^s g}{L^{q_1}(\R^d)}
\end{equation}
for both $D\in\{(-\triangle)^{s/2},(I-\triangle)^{s/2}\}$ and all Schwartz test functions $f,g\in\mathscr S(\R^d)$, where the implied constant may depend on all other parameters, but not on the functions $f$ and $g$.
\end{theorem}

On the left, we have a bilinear operator $\tau(f,g)=D^s(fg)$ from $\mathscr S(\R^d)\times\mathscr S(\R^d)$ into a quasi-normed space $L^r(\R^d)$; on the right, we have an upper bound consisting of a sum of two products of norms
\begin{equation}\label{eq:sumBound}
  \Norm{f}{X_0}\Norm{g}{Y_0}+\Norm{f}{X_1}\Norm{g}{Y_1},
\end{equation}
which could of course be dominated by the single product
\begin{equation*}
  \Norm{f}{X_0\cap X_1}\Norm{g}{Y_0\cap Y_1}:=(\Norm{f}{X_0}+\Norm{f}{X_1})(\Norm{g}{Y_0}+\Norm{g}{Y_1}),
\end{equation*}
but may in general be much smaller. Hence, the question of a vector-valued extension
\begin{equation}\label{eq:GraOhVec}
\begin{split}
  \Norm{x\mapsto D^s(\vec f\cdot\vec g)(x)}{L^r_x}
  &\overset{?}{\lesssim}\Norm{(x,y)\mapsto D^s \vec f(x)\cdot \vec g(y)}{L^{p_0}_x(L^{q_0}_y)} \\
  &\qquad+\Norm{(x,y)\mapsto \vec f(x)\cdot D^s\vec g(y)}{L^{p_1}_x(L^{q_1}_y)}
\end{split}
\end{equation}
of \eqref{eq:GraOh} will depend on understanding bilinear forms $\tau:X_0\cap X_1\times Y_0\cap Y_1\to Z$ with an upper bound involving two incomparable terms as in \eqref{eq:sumBound}. As a related side-remark, we mention that it seems convenient to develop the theory in (not necessarily complete) quasi-normed spaces rather than (complete by definition) quasi-Banach spaces: on the one hand, there seems to be little need for completeness in our essentially finite considerations; on the other hand, this framework allows us to directly apply the theory to the setting of Theorem \ref{thm:GraOh} with clearly incomplete spaces like
\begin{equation*}
    X_0=\mathscr S(\R^d),\quad\Norm{f}{X_0}=\Norm{D^s f}{L^{p_0}(\R^d)}
\end{equation*}
without a need to first worry about extending \eqref{eq:GraOh} from the initial test class to a complete space.

Besides the application to the vector-valued Kato--Ponce inequality, certainly of interest in its own right, studying the vector-valued extension of bilinear operators with the double norm bound \eqref{eq:sumBound} may serve as a toy model towards understanding the question of  convex body extensions of the far more complicated sparse bounds \eqref{eq:sparseDom} in the future.

We conclude this introduction by stating a version of our main result in the special case of $L^p$ spaces. Even the statement of the general case, to be proved in the body of the paper, will require additional preliminaries.

\begin{theorem}\label{thm:mainLp}
Let $p_k,q_k,r\in(0,\infty]$, let $F$ and $G$ be vector spaces, let $A_k:F\to L^{p_k}$ and $B_k:G\to L^{q_k}$ be linear operators, and
let $\tau:F\times G\to L^r$ is a bilinear operator with bound
\begin{equation}\label{eq:mainLpSca}
  \Norm{\tau(f,g)}{L^r}\lesssim \Norm{A_0 f}{L^{p_0}}\Norm{B_0g}{L^{q_0}}+\Norm{A_1f}{L^{p_1}}\Norm{B_1g}{L^{q_1}}.
\end{equation}
Then its vector-valued extension satisfies the bound
\begin{equation}\label{eq:mainLpVec}
  \Norm{\tau(\vec f,\vec g)}{L^r}\lesssim \sum_{k=0}^1 \Norm{(x,y)\mapsto A_k\vec f(x)\cdot B_k\vec g(y) }{L^{p_k}_x(L^{q_k}_y)},
\end{equation}
as well as each of the bounds where either or both $L^{p_k}_x(L^{q_k}_y)$ on the right is replaced by $L^{q_k}_y(L^{p_k}_x)$.
\end{theorem}

\begin{corollary}\label{cor:mainLp}\mbox{}
\begin{enumerate}
  \item\label{it:Holder} The vector-valued H\"older inequality \eqref{eq:HolderVec} holds for all $p,q,r$ as in \eqref{eq:HolderSca}.
  \item\label{it:KatoPonce} The vector-valued Kato--Ponce inequality \eqref{eq:GraOhVec} holds for all parameters as in the statement of Theorem \ref{thm:GraOh}.
\end{enumerate}
\end{corollary}

\begin{proof}\mbox{}
\begin{enumerate}
  \item With $p_k=p$, $q_k=q$, $F=L^p$, $G=L^q$, $A_k$ and $B_k$ taken to identity operators, and $\tau(f,g):=fg$ (the pointwise product), the assumption \eqref{eq:mainLpSca} is the classical H\"older inequality \eqref{eq:HolderSca}, and the conclusion \eqref{eq:mainLpVec} is the claimed vector-valued \eqref{eq:HolderVec}.
  \item We now take $F=G=\mathscr S(\R^d)$, $A_0=B_1=D^s$, $A_1=B_0=I$, and $\tau(f,g)=D^s(fg)$. Then assumption \eqref{eq:mainLpSca} is the scalar-valued Kato--Ponce inequality \eqref{eq:mainLpSca} of Theorem \ref{thm:mainLp}, and the conclusion \eqref{eq:mainLpVec} is the claimed vector-valued \eqref{eq:GraOhVec}.
\end{enumerate}
\end{proof}

From these illustrations, it should already be clear how to feed other scalar-valued inequalities into the machine, so we leave it to the reader to work out vector-valued versions of their favourite inequalities. Note in particular that, as in Corollary \ref{cor:mainLp}\eqref{it:Holder}, it often suffices to take $A_k=B_k=I$ in Theorem \ref{thm:mainLp}. The additional generality seems useful for covering examples like to Kato--Ponce inequality in Corollary \ref{cor:mainLp}\eqref{it:KatoPonce}.

The rest of this paper is structured as follows: In Section \ref{sec:tensor} we review some background about tensor products. In Section \ref{sec:convex}, we recall the relevant aspects of convex body domination and motivate our substitute notions applicable to the case of quasi-normed spaces. In the short Section \ref{sec:bilin1}, we obtain an abstract version of Theorem \ref{thm:mainLp} in the simpler case when there is only a single term on the right-hand side in both the assumption and in the conclusion. The final Section \ref{sec:bilin2} is then dedicated to lifting these considerations to the full generality of two-term upper bounds as in Theorem \ref{thm:mainLp}. The section culminates in an abstract version of Theorem \ref{thm:mainLp}, after which Theorem \ref{thm:mainLp} is a quick consequence.

\section{Tensor products of quasi-normed spaces}\label{sec:tensor}

As we hinted in the Introduction, we wish to set up an abstract framework for our analysis using the notion of tensor products of vector spaces. On the other hand, we wish to cover the case of quasi-normed spaces in general (i.e., allowing the possibility of a constant in the triangle inequality) and that of $L^p$ spaces for all $p>0$ in particular, so we cannot count on deriving any useful information out of the dual space. Various strategies of defining tensor products and related norms for quasi-Banach (and other) spaces are surveyed in \cite{Hansen:10}; however, for the larger part, the assumption of a dual space separating the points of $X$ is made, which clearly fails for our main example of $X=L^p$, $p\in(0,1)$‚ with $(L^p)^*=\{0\}$. On the other hand, our setting involves the special feature, different from the questions of typical interest in the theory of tensor product, that we are mainly interested in tensors
\begin{equation}\label{eq:fodotg}
  \omega=\vec f\odot\vec g=\sum_{i=1}^n f_i\otimes g_i,
\end{equation}
of (at most) a fixed finite length $n$. This is a significant advantage, and it therefore seems purposeful to develop the little theory that we need essentially from scratch, rather than relying on some heavy machinery built for completely different purposes.

To give a formal meaning to an expression like \eqref{eq:fodotg}, we follow \cite[Section 1.1]{Ryan:book}. Let first $X$ and $Y$ be real vector spaces (with no topology at this point, and therefore no continuity conditions in the definitions that follow). We denote by $X^{\sharp}$ and $Y^{\sharp}$ their {\em algebraic} duals, i.e., the vector spaces of linear functionals $x^{\sharp}:X\to\R$ and $y^{\sharp}:Y\to\R$. Let further $B(X,Y)$ be the space of bilinear forms $\beta:X\times Y\to\R$. For $x\in X$ and $y\in Y$, we define the elementary tensor $x\otimes y$ as a linear functional on $B(X,Y)$ by
\begin{equation*}
  (x\otimes y)(\beta):=\beta(x,y).
\end{equation*}
The (algebraic) tensor product $X\otimes Y$ is then the subspace of $B(X,Y)^{\sharp}$ spanned by the elementary tensors.

Let us then turn to discuss quasi-normed spaces $X$ and $Y$ with (possibly uninteresting) normed duals $X^*\subseteq X^{\sharp}$ and $Y^*\subseteq Y^{\sharp}$ of bounded linear functionals.
We already mentioned the injective norm
\begin{equation*}
   \Norm{\omega}{X\otimes_\eps Y}:=
   \sup\Big\{\sum_{i=1}^n\pair{f_i}{x^*}\pair{g_i}{y^*}:x^*\in\bar B_{X^*},y^*\in\bar B_{Y^*}\Big\},
\end{equation*}
which is only useful when the dual spaces are nontrivial, but we will also use of this norm to make some comparisons with earlier results in normed spaces.

The projective norm of $\omega\in X\otimes Y$ is
\begin{equation*}
   \Norm{\omega}{X\otimes_\pi Y}:=
   \inf\Big\{\sum_{i=1}^k\Norm{u_i}{X}\Norm{v_i}{Y}:\omega=\sum_{i=1}^k u_i\otimes v_i, u_i\in X,v_i\in Y, k\in\N\Big\},
\end{equation*}
where the infimum is taken over all possible representations of $\omega$, of any length. For any such representation, we have
\begin{equation*}
  \Babs{\sum_{i=1}^k \pair{u_i}{x^*}\pair{v_i}{y^*}}
  \leq\sum_{i=1}^k\Norm{u_i}{X}\Norm{v_i}{Y}\Norm{x^*}{X^*}\Norm{y^*}{Y^*}.
\end{equation*}
Taking an infimum over the different representations of $\omega$ and supremum over $x^*\in\bar B_{X^*},y^*\in\bar B_{Y^*}$ proves the well-known estimate
\begin{equation}\label{eq:injVsProj}
  \Norm{\omega}{X\otimes_\eps Y}\leq\Norm{\omega}{X\otimes_\pi Y}.
\end{equation}

The definition of the projective norm would seem to make sense for any quasi-normed spaces, but might still be trivial.

\begin{example}
Let $(S,\mathcal M,\mu)$ be a non-atomic measure space and $E\in\mathcal M$ have $\mu(E)\in(0,\infty)$. We write it as a disjoint union $E=\bigcup_{i=1}^k E_i$, where each $E_i$ has $\mu(E_i)=k^{-1}\mu(E)$. Then $1_E\otimes y=\sum_{i=1}^k 1_{E_i}\otimes y$, and hence
\begin{equation*}
  \Norm{1_E\otimes y}{L^p(\mu)\otimes_\pi Y}
  \leq\sum_{i=1}^k\Norm{1_{E_i}}{L^p(\mu)}\Norm{y}{Y}
  =\sum_{i=1}^k\mu(E_i)^{1/p}\Norm{y}{Y}
  =k^{1-1/p}\mu(E)^{1/p}\Norm{y}{Y}.
\end{equation*}
If $p\in(0,1)$, then $k^{1-1/p}\to 0$ with $k\to\infty$. Hence $\Norm{f\otimes y}{L^p\otimes_\pi Y}=0$ for every indicator $f=1_E$, and hence for every simple function $f\in L^p(\mu)$.
\end{example}

However, we will make some use of the following variant:
\begin{equation}\label{eq:nTensorNorm}
   \Norm{\omega}{X\otimes_\pi^n Y}:=
   \inf\Big\{\sum_{i=1}^n\Norm{u_i}{X}\Norm{v_i}{Y}:\omega=\sum_{i=1}^n u_i\otimes v_i, u_i\in X,v_i\in Y\Big\},
\end{equation}
where we only allow expansions of fixed length $n$ (or equivalently, of length at most $n$, since we can always add some $0\otimes 0$'s in the end). One should note that this is not even a quasi-norm; two tensors $\omega_i\in X\otimes Y$ with expansions of length at most $n$ and hence $\Norm{\omega_i}{X\otimes_\pi^n Y}<\infty$ may have a sum that cannot be written in such a form, and thus $\Norm{\omega_1+\omega_2}{X\otimes_\pi^n Y}=\infty$. Nevertheless, $X\otimes_\pi^n Y:=\{\omega\in X\otimes Y:\Norm{\omega}{X\otimes_\pi^n Y}<\infty\}$ is the image of $X^n\times Y^n$ under the bilinear map $(\vec f,\vec g)\mapsto\vec f\odot\vec g$, which makes it relevant for our considerations.

\section{Convex bodies and reducing operators}\label{sec:convex}

For $\vec x=\sum_{i=1}^n x_i\otimes\vec v_i\in X\otimes\R^n$ and $x^*\in X^*$ and $\vec v\in\R^n$, we define
\begin{equation*}
  \pair{\vec x}{x^*}=\sum_{i=1}^n \pair{x_i}{x^*}\vec v_i\in\R^n,\quad
  \vec x\cdot\vec v=\sum_{i=1}^n (\vec v_i\cdot \vec v)x_i\in X.
\end{equation*}

If $X$ is a Banach space with normed dual $X^*$, then
\begin{equation*}
  \Cave{\vec x}_X:=\{\pair{\vec x}{x^*}:x^*\in\bar B_{X^*}\}\subseteq\R^n
\end{equation*}
is a symmetric convex body. By the John ellipsoid theorem, there is a possibly degenerate ellipsoid $E_{\vec x}=A_{\vec x}\bar B_{\R^n}$, where $A_{\vec x}$ is non-negative self-adjoint operator on $\R^n$, such that
\begin{equation*}
  E_{\vec x}\subseteq\Cave{\vec x}_X\subseteq \sqrt{n}E_{\vec x}.
\end{equation*}
In principle, these considerations remain valid even if $X$ is just a quasi-Banach space, but the conclusions may be completely uninteresting, since it might happen that $X^*$ is too small or even $X^*=\{0\}$, as in the prominent case that $X=L^p(\mu)$ for $p\in(0,1)$ and a non-atomic measure space $(\Omega,\mathscr A,\mu)$. In this case, it follows that $\Cave{\vec x}_X=E_{\vec x}=\{0\}$, and $A_{\vec x}=0$. If we assume that $X$ is a quasi-Banach space {\em with a norming dual}, i.e, the property that
\begin{equation*}
  \Norm{x}{X}=\sup\{\pair{x}{x^*}:x^*\in\bar B_{X^*}\},
\end{equation*}
then the situation is essentially as good as in a Banach space.

On the other hand, even if $X$ is just a quasi-Banach space, we can define a semi-quasi-norm on $\R^n$ by
\begin{equation*}
  \Norm{\vec e}{X(\vec x)}:=\Norm{\vec{x}\cdot\vec e}{X}.
\end{equation*}
Recall that a quasi-norm means that we allow a constant in the triangle inequality; the prefix semi- means that we allow a nontrivial null space. We will shortly also need the notion of $p$-norm, which means that we have the $p$-triangle inequality $\Norm{x+y}{}^p\leq\Norm{x}{}^p+\Norm{y}{}^p$.

\begin{proposition}\label{prop:John}
Let $\Norm{\ }{}$ be a (semi-)quasi-norm on $\R^n$.  Then there is positive \mbox{(semi-)}definite matrix $A$ such that
\begin{equation*}
  \Norm{x}{}\approx \abs{Ax};
\end{equation*}
the implied constants depend only on $n$ and the quasi-triangle constant of $\Norm{\ }{}$.
\end{proposition}

\begin{proof}
By the Aoki--Rolewicz theorem, as stated e.g. in \cite[Theorem 1.2]{Malig04}, every quasi-norm $\Norm{\ }{}$ is equivalent to a $p$-norm $\Norm{\ }{}'$, where both $p\in(0,1]$ and the equivalence constants depend only on the quasi-triangle constant of $\|\ \|$. By an argument given in \cite[page 1237]{FR04}, every $p$-norm $\Norm{\ }{}'$ on $\R^n$ is equivalent to a norm of the form $x\mapsto\abs{Ax}$, where the equivalence constants depend only on $p$ and $n$. A combination of these equivalences proves the proposition in the case of a quasi-norm.

If $\Norm{\ }{}$ is only a semi-quasi-norm, then $W=\{x\in\R^n:\Norm{x}{}=0\}$ is a subspace, and $\Norm{\ }{}$ is a quasi-norm on its orthogonal complement $V=W^{\perp}$. The previous argument provides us with a positive definite linear mapping $A'$ on $V$ such that $\Norm{x}{}\approx\abs{A'x}$ for $x\in V$. If $P$ is the orthogonal projection of $\R^n$ onto $V$, then $(I-P)x\in W$, and the quasi-triangle inequality shows that $\Norm{x}{}\approx\Norm{Px}{}$ for all $x\in\R^n$. Thus $A=A'\circ P$ provides the desired semi-definite $A$.
\end{proof}

By Proposition \ref{prop:John}, we have
\begin{equation*}
  \Norm{\vec e}{X(\vec x)}\approx\abs{[\vec x]_{X}\vec e},
\end{equation*}
for some positive semi-definite operator $[\vec x]_X$ on $\R^n$. (The notation is slightly dangerous, since this operator is not uniquely determined; however, this ambiguity will not cause any real issues.) When $X=L^p(Q;\R^n)$ and $X\otimes\R^n$ is identified with $L^p(Q;\R^{n\times n})$ so that $\vec w\cdot\vec e=W\vec e$ for a matrix-valued function $W$, this idea is commonly used in the theory of matrix-weights, and a non-negative matrix $[W]_{L^p(Q;\R^n)}$ satisfying
\begin{equation*}
  \abs{[W]_{L^p(Q;\R^n)}\vec e}\approx\Norm{W\vec e}{L^p(Q;\R^n)}
\end{equation*}
is called a {\em reducing matrix} (or operator) associated with the matrix weight $W$. What we have defined above is a natural extension of this idea to general quasi-Banach spaces, and we keep referring to $[\vec x]_X$ as the reducing matrix.

In the case of a Banach space, we may invoke duality to see that
\begin{equation*}
\begin{split}
   \Norm{\vec e}{X(\vec x)} &=\sup\{\pair{\vec{x}\cdot\vec e}{x^*}:x^*\in\bar B_{X^*}\} 
   =\sup\{\pair{\vec{x}}{x^*}\cdot\vec e:x^*\in\bar B_{X^*}\} \\
   &=\sup\{\vec v\cdot\vec e:\vec v\in\Cave{\vec x}_{X}\} 
   \approx\sup\{\vec v\cdot\vec e:\vec v\in E_{\vec x}\} \\
   &=\sup\{A_{\vec x}\vec u\cdot\vec e:\vec u\in \bar B_{\R^n}\} 
   =\sup\{\vec u\cdot A_{\vec x}\vec e:\vec u\in \bar B_{\R^n}\} 
   =\abs{A_{\vec x}\vec e}.
\end{split}
\end{equation*}
Thus $\abs{A_{\vec x}\vec e}\approx\abs{[\vec x]_X\vec e}$ for all $\vec e\in\R^n$, so also $A_{\vec x}$ defines, up to constants, the same norm as $[\vec x]_X$. This is natural, since both these matrices essentially arise from the John ellipsoid theorem, by considerations essentially dual to each other. When $X$ is a normed space, we are free to switch from one framework to the other through duality. Once we move to quasi-normed spaces, however, only the reducing matrix $[\vec x]_X$ remains well defined, and we choose it as the replacement of the convex body $\Cave{\vec x}_X$ and its John ellipsoid $E_{\vec x}=A_{\vec x}\bar B_{\R^n}$.

The following result shows that the Minkowski dot product of two convex bodies, $\Cave{\vec x}_X\cdot\Cave{\vec x}_Y$ (strictly speaking an interval $[-c,c]$, but identified with its right end point $c$), the quantity featuring in the recent convex body domination literature, admits several other equivalent formulations in the setting of normed spaces. One of these equivalent quantities is the norm of the product of the reducing matrices, $\abs{[\vec x]_X[\vec y]_Y}$, which motivates taking this as the replacement of $\Cave{\vec x}_X\cdot\Cave{\vec x}_Y$ when quasi-norms are involved.

\begin{proposition}\label{prop:equiNormsBanach}
Let $X$ and $Y$ be normed spaces, $\vec x\in X\otimes\R^n$, and $\vec y\in Y\otimes\R^n$. Then
\begin{equation*}
  \abs{[\vec x]_X[\vec y]_Y}
  \lesssim\Cave{\vec x}_X\cdot\Cave{\vec y}_Y
  =\Norm{\vec x\odot\vec y}{X\otimes_\eps Y}
  \leq\Norm{\vec x\odot\vec y}{X\otimes_\pi Y}
  \lesssim\abs{[\vec x]_X[\vec y]_Y},
\end{equation*}
where the implied constants depend only on $n$.
\end{proposition}

\begin{proof}
\begin{equation*}
\begin{split}
  \abs{[\vec x]_X[\vec y]_Y}
  &=\sup\{\abs{[\vec x]_X[\vec y]_Y\vec e}:\vec e\in\bar B_{\R^n}\} \\
  &\approx\sup\{\Norm{\vec x\cdot[\vec y]_Y\vec e}{X}:\vec e\in\bar B_{\R^n}\} \\
  &=\sup\{\pair{\vec x\cdot[\vec y]_Y\vec e}{x^*}:\vec e\in\bar B_{\R^n},x^*\in\bar B_{X^*}\} \\
  &=\sup\{[\vec y]_Y\pair{\vec x}{x^*}\cdot\vec e:\vec e\in\bar B_{\R^n},x^*\in\bar B_{X^*}\} \\
  &=\sup\{\abs{[\vec y]_Y\pair{\vec x}{x^*}}:x^*\in\bar B_{X^*}\} \\
  &\approx\sup\{\Norm{\vec y\cdot\pair{\vec x}{x^*}}{Y}:x^*\in\bar B_{X^*}\} \\
  &=\sup\{\Bpair{\vec y\cdot\pair{\vec x}{x^*}}{y^*}:x^*\in\bar B_{X^*},y^*\in\bar B_{Y^*}\} \\
  &=\sup\{\pair{\vec x}{x^*}\cdot\pair{\vec y}{y^*}:x^*\in\bar B_{X^*},y^*\in\bar B_{Y^*}\} 
  =\Cave{\vec x}_X\cdot\Cave{\vec y}_Y \\
  &=\sup\{\pair{\vec x\odot\vec y}{x^*\otimes y^*}:x^*\in \bar B_{X^*},y^*\in\bar B_{Y^*}\}
  =\Norm{\vec x\odot\vec y}{X\otimes_\eps Y}
\end{split}
\end{equation*}
by definition of the convex bodies and the injective tensor norm on the last two lines.
The estimate $ \Norm{\omega}{X\otimes_\eps Y}\leq\Norm{\omega}{X\otimes_\pi Y}$ is a basic fact about tensor products recalled in \eqref{eq:injVsProj}.

To prove the remaining estimate, from the defining identity
\begin{equation*}
  \Norm{\vec x\cdot\vec e}{X}\approx\abs{[\vec x]_X\vec e}
\end{equation*}
we see that $\vec x\cdot\vec e=0$ for all $\vec e\in\kernel([\vec x]_X)$. If $P_{\range([\vec x]_X)}$ is the orthogonal projection onto $\range([\vec x]_X)$, we have $(I-P_{\range([\vec x]_X)})\vec e\in\kernel([\vec x]_X)$ for all $\vec e\in\R^n$, hence
\begin{equation*}
  (I-P_{\range([\vec x]_X)})\vec x\cdot\vec e
  =\vec x\cdot (I-P_{\range([\vec x]_X)})\vec e=0
\end{equation*}
for all $\vec e\in\R^n$, and thus $\vec x=P_{\range([\vec x]_X)}\vec x$. It follows that
\begin{equation*}
  [\vec x]_X^{-1}\vec x=[\vec x]_X^{-1}P_{\range([\vec x]_X)}\vec x
\end{equation*}
is well defined, and we can write
\begin{equation*}
  \vec x\odot\vec y
  =[\vec x]_X^{-1}\vec x\odot [\vec x]_X\vec y
  =\sum_{i=1}^n(\vec e_i\cdot [\vec x]_X^{-1}\vec x)\otimes (\vec e_i\cdot[\vec x]_X\vec y).
\end{equation*}
Hence
\begin{equation*}
  \Norm{\vec x\odot\vec y}{X\otimes_{\pi}Y}
  \leq\sum_{i=1}^n \Norm{\vec e_i\cdot [\vec x]_X^{-1}\vec x}{X}\Norm{\vec e_i\cdot[\vec x]_X\vec y}{Y}.
\end{equation*}
Here
\begin{equation*}
\begin{split}
  \Norm{\vec e_i\cdot [\vec x]_X^{-1}\vec x}{X}
  &=\Norm{[\vec x]_X^{-1}P_{\range([\vec X]_X)}\vec e_i\cdot \vec x}{X} \\
  &\approx\abs{[\vec x]_X[\vec x]_X^{-1}P_{\range([\vec X]_X)}\vec e_i}
  =\abs{P_{\range([\vec X]_X)}\vec e_i}\leq 1,
\end{split}
\end{equation*}
and
\begin{equation*}
  \Norm{\vec e_i\cdot[\vec x]_X\vec y}{Y}
  =\Norm{[\vec x]_X\vec e_i\cdot\vec y}{Y}
  \approx\abs{[\vec y]_Y[\vec x]_X\vec e_i}
  \leq\abs{[\vec y]_Y[\vec x]_X}=\abs{[\vec x]_X[\vec y]_Y}.
\end{equation*}
Thus
\begin{equation*}
   \Norm{\vec x\odot\vec y}{X\otimes_{\pi}Y}\lesssim\abs{[\vec x]_X[\vec y]_Y}.
\end{equation*}
\end{proof}

While definable in a rather general setting, an inconvenience of the quantity $\abs{[\vec x]_X[\vec y]_Y}$ is that, unlike the other quantities in Proposition \ref{prop:equiNormsBanach}, it is not in any obvious way a function of the tensor $\vec x\odot\vec y$, but would seem to depend on the particular $\vec x$ and $\vec y$ in its representation. The next result partially addresses this concern, but only at the cost of using the special ``norm'' \eqref{eq:nTensorNorm} which, as we observed, is not even a quasi-norm.

\begin{proposition}\label{prop:equiNormsQBanach}
Let $X$ and $Y$ be quasi-normed spaces, $\vec x\in X\otimes\R^n$, and $\vec y\in Y\otimes\R^n$. Then
\begin{equation*}
  \abs{[\vec x]_X[\vec y]_Y}
  \approx\Norm{\vec x\odot\vec y}{X\otimes_\pi^n Y}
\end{equation*}
where the implied constants depend only on $n$ and the quasi-triangle constants of $X$ and $Y$.
\end{proposition}

\begin{proof}
The proof that $\Norm{\vec x\odot\vec y}{X\otimes_\pi^n Y}\lesssim\abs{[\vec x]_X[\vec y]_Y}$ is the same as the corresponding estimate for Banach spaces, just observing that the proof there already expressed $\vec x\odot\vec y$ as a sum of $n$ elementary tensors $(\vec e_i\cdot [\vec x]_X^{-1}\vec x)\otimes (\vec e_i\cdot[\vec x]_X\vec y)$ only. (The dependence on the quasi-triangle constants comes from the fact that such dependence may be present in the estimates of the form $\Norm{\vec x\cdot\vec u}{X}\approx\abs{[\vec x]_X\vec u}$.) As for the other direction, we first observe that, for any orthonormal basis $(\vec e_i)_{i=1}^n$ of $\R^n$, we have
\begin{equation*}
\begin{split}
  \abs{[\vec x]_X[\vec y]_Y}
  &=\sup\{\abs{[\vec x]_X[\vec y]_Y\vec e}:\vec e\in\bar B_{\R^n}\} \\
  &\approx\sup\{\Norm{\vec x\cdot[\vec y]_Y\vec e}{X}:\vec e\in\bar B_{\R^n}\} \\
  &\lesssim\sum_{i=1}^n\sup\{\Norm{(\vec x\cdot\vec e_i)(\vec e_i\cdot [\vec y]_Y\vec e)}{X}:\vec e\in\bar B_{\R^n}\} \\
  &\leq\sum_{i=1}^n\Norm{\vec x\cdot\vec e_i}{X}\abs{[\vec y]_Y\vec e_i} \\
  &\approx\sum_{i=1}^n\Norm{\vec x\cdot\vec e_i}{X}\Norm{\vec y\cdot\vec e_i}{Y}.
\end{split}
\end{equation*}
Here, on the right, we have $\sum_{i=1}^n\Norm{x_i}{X}\Norm{y_i}{Y}$ related to a particular expansion $\sum_{i=1}^n x_i\otimes y_i=\vec x\odot\vec y$.
However, for $\Norm{\vec x\odot\vec y}{X\otimes_\pi^n Y}$, we need to consider all possible expansions. To this end, suppose that $\vec x\odot\vec y=\vec z\odot\vec w$. Let $u_1,\ldots,u_k\in X$ be linearly independent elements with $$\lspan\{u_1,\ldots,u_k\}=\lspan\{x_1,\ldots,x_n,z_1,\ldots,z_n\},$$ and
$v_1,\ldots,v_\ell\in Y$ be linearly independent elements with $$\lspan\{v_1,\ldots,v_\ell\}=\lspan\{y_1,\ldots,y_n,w_1,\ldots,w_n\};$$ thus $k,\ell\leq 2n$.

Then
\begin{equation*}
  \vec x=A\vec u,\quad\vec y=B\vec v,\quad\vec z=C\vec u,\quad\vec w=D\vec v
\end{equation*}
for some $n\times k$ matrices $A,C$ and $n\times\ell$ matrices $B,D$.
Since $u_i$ are linearly independent in $X$ and $v_j$ in $Y$, the elementary tensors $u_i\otimes v_j$ are linearly independent in $X\otimes Y$ \cite[Proposition 1.1(a)]{Ryan:book}.
Then the condition
\begin{equation*}
\begin{split}
  &\vec x\odot\vec y=(A\vec u)\odot(B\vec v)=(B^*A)\vec u\odot\vec v \\
= &\vec z\odot\vec w=(C\vec u)\odot(D\vec v)=(D^*C)\vec u\odot\vec v,
\end{split}
\end{equation*}
implies, after elementary coefficient-chasing, that $B^*A=D^*C$.

Moreover, we observe that
\begin{equation*}
  \abs{[A\vec u]_X\vec e}\approx\Norm{A\vec u\cdot\vec e}{X}
  =\Norm{\vec u\cdot A^*\vec e}{X}
  \approx\abs{[\vec u]A^*\vec e}.
\end{equation*}
Hence
\begin{equation*}
\begin{split}
  \abs{[A\vec u]_X[B\vec v]_Y}
  &=\sup\{\abs{[A\vec u]_X[B\vec v]_Y\vec e}:\vec e\in\bar B_{\R^n}\} \\
  &\approx\sup\{\abs{[\vec u]_XA^*[B\vec v]_Y\vec e}:\vec e\in\bar B_{\R^n}\} \\
  &=\abs{[\vec u]_XA^*[B\vec v]_Y} \\
  &=\abs{([\vec u]_XA^*[B\vec v]_Y)^*} \\
  &=\abs{[B\vec v]_Y A[\vec u]_X} \\
  &=\sup\{\abs{[B\vec v]_Y A[\vec u]_X\vec e}:\vec e\in\bar B_{\R^n}\} \\
  &\approx\sup\{\abs{[\vec v]_Y B^*A[\vec u]_X\vec e}:\vec e\in\bar B_{\R^n}\} \\
  &=\abs{[\vec v]_Y B^*A[\vec u]_X}.
\end{split}
\end{equation*}
It follows that
\begin{equation*}
\begin{split}
  \abs{[\vec x]_X[\vec y]_Y}
  &=\abs{[A\vec u]_X[B\vec v]_Y} \\
  &\approx\abs{[\vec v]_Y B^*A[\vec u]_X} \\
  &=\abs{[\vec v]_Y D^*C[\vec u]_X} \\
  &\approx\abs{[C\vec u]_X [D\vec v]_Y} \\
  &=\abs{[\vec z]_X[\vec w]_Y} \\
  &\lesssim\sum_{i=1}^n\Norm{z_i}{X}\Norm{w_i}{Y}
\end{split}
\end{equation*}
whenever
\begin{equation*}
  \vec x\odot\vec y=\vec z\odot\vec w=\sum_{i=1}^n(\vec z\cdot\vec e_i)\otimes(\vec e_i\cdot\vec w_i)
  =\sum_{i=1}^n z_i\otimes w_i.
\end{equation*}
Taking the infimum over all such representations, we deduce that
\begin{equation*}
  \abs{[\vec x]_X[\vec y]_Y}\lesssim\Norm{\vec x\odot\vec y}{X\otimes_\pi^n Y}.
\end{equation*}
\end{proof}

We will finally discuss a more concrete situation, where $X$ and $Y$ have appropriate function space structure. As it turns out, this allows the identification of  $\abs{[\vec x]_X[\vec y]_Y}$ with more natural tensor (quasi-)norms.

\begin{definition}\label{def:lattice}
We say that $X$ is a {\em quasi-normed function lattice} if $X\subseteq L^0(S,\mathcal M)$ is a vector space of measurable functions equipped with a quasi-norm $\Norm{\ }{X}$ such that: whenever $f\in X$ and $h\in L^0(S,\mathcal M)$ satisfies $\abs{h}\leq\abs{f}$ pointwise, then $h\in X$ and $\Norm{h}{X}\leq\Norm{f}{X}$.

If $X\subseteq L^0(S,\mathcal M)$ is a quasi-normed function lattice and $Y$ is a quasi-normed space, we define on $X\otimes Y$ the norm
\begin{equation*}
  \BNorm{\sum_{i=1}^n x_i\otimes y_i}{X(Y)}:=\BNorm{s\mapsto \Norm{\sum_{i=1}^n x_i(s)y_i}{Y}}{X}.
\end{equation*}

We say that a quasi-normed space $X$ is {\em related to a quasi-normed function lattice} $E\subseteq L^0(S,\mathcal M)$ if there is a linear operator $J:X\to E$ such
\begin{equation*}
  \Norm{x}{X}=\Norm{Jx}{E}\quad\forall x\in X.
\end{equation*}
We extend the definition of $X(Y)$ to this situation by setting
\begin{equation*}
  \BNorm{\sum_{i=1}^n x_i\otimes y_i}{X(Y)}:=
  \BNorm{\sum_{i=1}^n Jx_i\otimes y_i}{E(Y)}=
  \BNorm{s\mapsto \Norm{\sum_{i=1}^n Jx_i(s)y_i}{Y}}{E}.
\end{equation*}
\end{definition}

In these situations, we have the following:

\begin{proposition}\label{prop:equiNormsQBanachFn}
Let $X$ and $Y$ be quasi-normed spaces, and suppose in addition that $X$ (resp. $Y$) is a quasi-normed function lattice.
Then for all  $\vec x\in X\otimes\R^n$, and $\vec y\in Y\otimes\R^n$, we have
\begin{equation*}
  \abs{[\vec x]_X[\vec y]_Y}
  \approx\begin{cases}\Norm{\vec x\odot\vec y}{X(Y)}, & \text{resp.} \\
  \Norm{\vec x\odot\vec y}{Y(X)},\end{cases}
\end{equation*}
where the implied constants depend only on $n$ and the quasi-triangle constants of $X$ and $Y$. In particular, if both $X$ and $Y$ be quasi-normed function lattices, then
\begin{equation*}
  \abs{[\vec x]_X[\vec y]_Y}
  \approx\Norm{\vec x\odot\vec y}{X(Y)}
  \approx\Norm{\vec x\odot\vec y}{Y(X)}.
\end{equation*}
\end{proposition}

\begin{proof}
We consider the case that $X$ is a quasi-normed function lattice. The other case follows by symmetry, since the left-hand side is symmetric in $X$ and $Y$, since the self-adjoint matrices $A=[\vec x]_X, B=[\vec y]_Y$ satisfy
\begin{equation*}
  \abs{AB}=\abs{(AB)^*}=\abs{B^*A^*}=\abs{BA}.
\end{equation*}
We have
\begin{equation*}
\begin{split}
  \abs{[\vec x]_X[\vec y]_Y}
  &=\sup\{\abs{[\vec x]_X[\vec y]_Y\vec e}:\vec e\in\bar B_{\R^n}\} \\
  &\approx\sup\{\Norm{\vec x\cdot[\vec y]_Y\vec e}{X}:\vec e\in\bar B_{\R^n}\} \\
  &=\sup\Big\{\BNorm{s\mapsto\abs{[\vec y]_Y\vec x(s)\cdot\vec e}}{X}:\vec e\in\bar B_{\R^n}\Big\} \\
  &\leq\BNorm{s\mapsto \abs{[\vec y]_Y\vec x(s)}}{X} \\
  &\approx\BNorm{s\mapsto \Norm{\vec y\cdot\vec x(s)}{Y} }{X} \\
  &=\Norm{\vec x\cdot\vec y}{X(Y)}.
\end{split}
\end{equation*}
In the other direction, we have
\begin{equation*}
\begin{split}
  \Norm{\vec x\cdot\vec y}{X(Y)}
  &\approx\BNorm{s\mapsto \abs{[\vec y]_Y\vec x(s)}}{X} \\
  &\leq\BNorm{s\mapsto \Babs{\sum_{i=1}^n\vec e_i (\vec e_i\cdot[\vec y]_Y\vec x(s))}}{X} \\
  &\leq\BNorm{s\mapsto \sum_{i=1}^n\abs{[\vec y]_Y\vec e_i\cdot\vec x(s)}}{X} \\
  &\lesssim\sum_{i=1}^n\Norm{s\mapsto \abs{[\vec y]_Y\vec e_i\cdot\vec x(s)}}{X} \\
  &=\sum_{i=1}^n\Norm{[\vec y]_Y\vec e_i\cdot\vec x}{X} \\
  &\approx\sum_{i=1}^n\abs{[\vec x]_X[\vec y]_Y\vec e_i} \\
  &\lesssim\abs{[\vec x]_X[\vec y]_Y}.
\end{split}
\end{equation*}
\end{proof}

\begin{corollary}\label{cor:equiNormsQBanachFn}
Let $X$ and $Y$ be quasi-normed spaces, where $X$ (resp. $Y$) is related to a quasi-normed function lattice.
Then for all  $\vec x\in X\otimes\R^n$, and $\vec y\in Y\otimes\R^n$, we have
\begin{equation*}
  \abs{[\vec x]_X[\vec y]_Y}
  \approx\begin{cases}
  \Norm{\vec x\odot\vec y}{X(Y)}, & \text{resp.} \\
  \Norm{\vec x\odot\vec y}{Y(X)}. \end{cases}
\end{equation*}
where the implied constants depend only on $n$ and the quasi-triangle constants of $X$, $Y$, and the quasi-normed function lattice related to one of these spaces. In particular, if both $X$ and $Y$ are related to a quasi-normed function lattice, then
\begin{equation*}
  \abs{[\vec x]_X[\vec y]_Y}
  \approx\Norm{\vec x\odot\vec y}{X(Y)}
  \approx\Norm{\vec x\odot\vec y}{Y(X)}.
\end{equation*}
\end{corollary}

\begin{proof}
We consider the case that $X$ is related to a quasi-normed function lattice $E$ via a linear operator $J:X\to E$ such that $\Norm{x}{X}=\Norm{Jx}{E}$ for all $x\in X$; the other cases follow by symmetry.
Then for each $\vec e\in\R^n$, we have
\begin{equation*}
  \abs{[\vec x]_X\vec e}
  \approx\Norm{\vec x\cdot\vec e}{X}
  =\Norm{J(\vec x\cdot\vec e)}{E}=\Norm{(J\vec x)\cdot\vec e}{E}
  \approx\abs{[J\vec x]_E\vec e}.
\end{equation*}
Hence
\begin{equation*}
\begin{split}
    \abs{[\vec x]_X[\vec y]_Y}
    &=\sup\{\abs{[\vec x]_X[\vec y]_Y\vec e}:\vec e\in\bar B_{\R^n}\} \\
    &\approx\sup\{\abs{[J\vec x]_E[\vec y]_Y\vec e}:\vec e\in\bar B_{\R^n}\} 
    =\abs{[J\vec x]_E[\vec y]_Y}.
\end{split}
\end{equation*}
Thus Proposition \ref{prop:equiNormsQBanachFn} guarantees that
\begin{equation*}
 \abs{[\vec x]_X[\vec y]_Y}
 \approx\abs{[J\vec x]_E[\vec y]_Y}
 \approx\Norm{J\vec x\odot\vec y}{E(Y)}
 =\Norm{\vec x\odot\vec y}{X(Y)},
\end{equation*}
where the last step is simply the definition of $\Norm{\ }{X(Y)}$.
\end{proof}

\section{Bilinear operators with a simple bound}\label{sec:bilin1}

Having studied the properties of the quantity $\abs{[\vec x]_X[\vec y]_Y}$ in its own right, we now show that this quantity also qualifies for an upper bound in vector-valued extensions of bilinear operators; this is the ultimate justification that this is the correct generalisation of the $\Cave{\vec x}_X\cdot\Cave{\vec y}_Y$.

\begin{proposition}\label{prop:normToCave}
Let $X,Y,Z$ be quasi-normed spaces and
let $\tau:X\times Y\to Z$ be a bounded bilinear operator.
Then for all $\vec f\in X\otimes\R^n$ and $\vec g\in Y\otimes\R^n$, we have
\begin{equation*}
  \Norm{\tau(\vec f,\vec g)}{Z}\lesssim\Norm{\tau}{}\abs{[\vec x]_{X}[\vec y]_{Y}}
\end{equation*}
where the implied constant depends only on $n$ and the quasi-triangle constants of $X,Y,Z$.
\end{proposition}

\begin{remark}
The case when $X,Y$ are normed spaces $Z=\R$ is \cite[Proposition 4.2]{Hyt:convex}. The case when $X,Y,Z$ is are normed spaces would be a direct consequence of that special case via
\begin{equation*}
  \Norm{\tau(\vec f,\vec g)}{Z}=\sup_{z^*\in\bar B_{Z^*}}\abs{\pair{\tau(\vec f,\vec g)}{z^*}},
\end{equation*}
where each $\pair{\tau(\cdot,\cdot)}{z^*}$ is a bilinear form $X\times Y\to\R$ of norm at most $\Norm{\tau}{}$. The general case does not seem to allow such a simple reduction, since it may happen that $Z^*=\{0\}$ e.g. in the important special case that $Z=L^p$ for $p\in(0,1)$. Nevertheless, the direct proof in the general case below is based on the same idea.
\end{remark}

\begin{proof}
The proof follows the same pattern as that of the estimate $\Norm{\vec x\odot\vec y}{X\otimes_\pi Y}\lesssim\abs{[\vec x]_X[\vec y]_Y}$ in Proposition \ref{prop:equiNormsBanach}.
We have
\begin{equation*}
  \vec x\odot\vec y
  =[\vec x]_X^{-1}\vec x\odot [\vec x]_X\vec y
  =\sum_{i=1}^n(\vec e_i\cdot [\vec x]_X^{-1}\vec x)\otimes (\vec e_i\cdot[\vec x]_X\vec y).
\end{equation*}
Hence,
\begin{equation*}
  \tau(\vec x,\vec y)
  =\sum_{i=1}^n\tau(\vec e_i\cdot [\vec x]_X^{-1}\vec x,\vec e_i\cdot[\vec x]_X\vec y),
\end{equation*}
and the quasi-triangle inequality in $Z$ gives
\begin{equation*}
  \Norm{\tau(\vec x,\vec y)}{Z}
  \lesssim\sum_{i=1}^n \Norm{\tau}{} \Norm{\vec e_i\cdot [\vec x]_X^{-1}\vec x}{X}\Norm{\vec e_i\cdot[\vec x]_X\vec y}{Y},
\end{equation*}
where the implied constant may depend on $n$ and the quasi-triangle constant of $Z$.

Here, as in the proof of Proposition \ref{prop:equiNormsBanach},
\begin{equation*}
\begin{split}
  \Norm{\vec e_i\cdot [\vec x]_X^{-1}\vec x}{X}
  &=\Norm{[\vec x]_X^{-1}P_{\range([\vec X]_X)}\vec e_i\cdot \vec x}{X} \\
  &\approx\abs{[\vec x]_X[\vec x]_X^{-1}P_{\range([\vec X]_X)}\vec e_i}
  =\abs{P_{\range([\vec X]_X)}\vec e_i}\leq 1,
\end{split}
\end{equation*}
where the implied constant may depend on $n$ and the quasi-triangle constant of $X$, 
and
\begin{equation*}
  \Norm{\vec e_i\cdot[\vec x]_X\vec y}{Y}
  =\Norm{[\vec x]_X\vec e_i\cdot\vec y}{Y}
  \approx\abs{[\vec y]_Y[\vec x]_X\vec e_i}
  \leq\abs{[\vec y]_Y[\vec x]_X}=\abs{[\vec x]_X[\vec y]_Y},
\end{equation*}
where the implied constant may depend on $n$ and the quasi-triangle constant of $Y$. 
Altogether,
\begin{equation*}
  \Norm{\tau(\vec x,\vec y)}{Z}\lesssim\abs{[\vec x]_X[\vec y]_Y}.
\end{equation*}
where the implied constant may depend on $n$ and the quasi-triangle constants of $X,Y,Z$. 
\end{proof}

\section{Bilinear operators with a two-term bound}\label{sec:bilin2}

We now turn to the more complicated situation of two-term bounds as in the Kato--Ponce inequality discussed in the Introduction. A key to this is the following finite-dimensional lemma, showing that a bilinear form with a two-term bound can always be split into a sum of two bilinear forms, each of which is dominated by one of the terms only.

\begin{lemma}\label{lem:bound2eq}
Let $Z$ be a quasi-normed space and $\tau:\R^m\times\R^n\to Z$ be a bilinear operator such that, for some matrices $A_0,A_1,B_0,B_1$ we have
\begin{equation*}
  \Norm{\tau(x,y)}{Z}\leq\abs{A_0x}\abs{B_0y}+\abs{A_1x}\abs{B_1y}
\end{equation*}
for all vectors $x,y$. Then there are bilinear forms $\tau_0,\tau_1:\R^m\times\R^n\to Z$ such that $\tau=\tau_0+\tau_1$ and
\begin{equation*}
  \Norm{\tau_k(x,y)}{Z}\lesssim\abs{A_kx}\abs{B_ky},\quad k=0,1.
\end{equation*}
\end{lemma}

\begin{proof}
We proceed through several cases of increasing generality.

\subsubsection*{Case: $A_0=B_0=I$ and $A_1,B_1$ are positive semi-definite:}
In this case we have
\begin{equation*}
  A_1x=\sum_{i=1}^m \alpha_i e_i\pair{x}{e_i},\quad
  B_1y=\sum_{j=1}^n \beta_j f_j\pair{y}{f_j},
\end{equation*}
where $(e_i)_{i=1}^m$ and $(f_j)_{j=1}^n$ are orthonormal and $\alpha_i,\beta_j\geq 0$.

Let $x_i:=\pair{x}{e_i}$, $y_j:=\pair{y}{f_j}$.
Then
\begin{equation*}
   \abs{A_1x}\abs{B_1y}
   =\Big(\sum_{i=1}^m\alpha_i^2\abs{x_i}^2\Big)^{1/2} \Big(\sum_{j=1}^n\beta_j^2\abs{y_j}^2\Big)^{1/2}.
\end{equation*}
With $x=e_i$ and $y=f_j$ the assumption then implies that
\begin{equation*}
  \Norm{\tau_{ij}}{Z}:=\Norm{\tau(e_i,f_j)}{Z}\leq 1+\abs{A_1e_i}\abs{B_1e_j}=1+\alpha_i\beta_j.
\end{equation*}
We define
\begin{equation*}
\begin{split}
  \tau_0(x,y)&:=\sum_{i=1}^m\sum_{j=1}^n\pair{x}{e_i}\frac{\tau_{ij}}{1+\alpha_i\beta_j}\pair{y}{f_j} \\
  \tau_1(x,y)&:=\sum_{i=1}^m\sum_{j=1}^n\pair{x}{e_i}\alpha_i\frac{\tau_{ij}}{1+\alpha_i\beta_j}\beta_j\pair{y}{f_j}.
\end{split}
\end{equation*}
Then
\begin{equation*}
\begin{split}
  \tau_0(x,y)+\tau_1(x,y)
  &=\sum_{i=1}^m\sum_{j=1}^n\pair{x}{e_i}\tau_{ij}\pair{y}{f_j} \\
  &=\tau\Big(\sum_{i=1}^m\pair{x}{e_i}e_i,\sum_{j=1}^n\pair{y}{f_j}f_j\Big)=\tau(x,y).
\end{split}
\end{equation*}
From the quasi-triangle inequality, we have
\begin{equation*}
  \Norm{\tau_0(x,y)}{Z}
  \lesssim\sum_{i=1}^m\sum_{j=1}^n\abs{x_i}\Norm{\frac{\tau_{ij}}{1+\alpha_i\beta_j}}{Z}\abs{y_j}
  \leq\sum_{i=1}^m\sum_{j=1}^n\abs{x_i}\abs{y_j}
  \lesssim\abs{x}\abs{y}
\end{equation*}
and similarly
\begin{equation*}
  \Norm{\tau_1(x,y)}{Z}
  \lesssim 
  \sum_{i=1}^m\sum_{j=1}^n\abs{\alpha_i x_i}\abs{\beta_j y_j}
  =\abs{A_1x}\abs{B_1y}.
\end{equation*}
Hence we obtained the claim in the case under consideration.

\subsubsection*{Case: $A_0=B_0=I$ and $A_1,B_1$ general}
Consider the singular value decompositions
\begin{equation*}
  A_1x=\sum_{i=1}^m \alpha_i \tilde e_i\pair{x}{e_i},\quad
  B_1y=\sum_{j=1}^n \beta_j \tilde f_j\pair{y}{f_j},
\end{equation*}
where each of the systems $(e_i)_{i=1}^m$, $(\tilde e_i)_{i=1}^m$, $(f_i)_{i=1}^n$, and $(\tilde f_i)_{i=1}^n$ is orthonormal and $\alpha_i,\beta_j\geq 0$.
Let
\begin{equation*}
  \tilde A_1x=\sum_{i=1}^m \alpha_i e_i\pair{x}{e_i},\quad
  \tilde B_1y=\sum_{j=1}^n \beta_j f_j\pair{y}{f_j}.
\end{equation*}
These are positive semi-definite matrices with $\abs{A_1x}=\abs{\tilde A_1x}$ and $\abs{B_1y}=\abs{\tilde B_1y}$, so both the assumptions and the conclusions remain equivalent with $\tilde A_1,\tilde B_1$ in place of $A_1,B_1$. Since we already proved the case with non-negative self-adjoint  $\tilde A_1,\tilde B_1$, we also get the case general $A_1,B_1$ by the said equivalence.

\subsubsection*{Case: $A_0,B_0$ invertible and $A_1,B_1$ general}
Let us denote $u:=A_0x$ and $v=B_0y$, hence $x=A_0^{-1}u$ and $y=B_0^{-1}v$. Thus the assumption can be equivalently written as
\begin{equation*}
  \abs{\tau(A_0^{-1}u,B_0^{-1}v)}\leq\abs{u}\abs{v}+\abs{A_1A_0^{-1}u}\abs{B_1B_+0^{-1}v}.
\end{equation*}
Hence the new bilinear form $\tilde\tau:=\tau(A_0^{-1}\cdot,B_0^{-1}\cdot)$ and the matrices $\tilde A_1:=A_1A_0^{-1}$ and $\tilde B_1:=B_1B_0^{-1}$ satisfy
\begin{equation*}
  \abs{\tilde\tau(u,v)}\leq\abs{u}\abs{v}+\abs{\tilde A_1u}\abs{\tilde B_1v},
\end{equation*}
which is exactly as in the previous case. By that case, we find bilinear forms $\tilde\tau_0,\tilde\tau_1$ such that $\tilde\tau=\tilde\tau_0+\tilde\tau_1$ and
\begin{equation*}
  \abs{\tilde\tau_0(u,v)}\lesssim\abs{u}\abs{v},\quad
  \abs{\tilde\tau_1(u,v)}\lesssim\abs{\tilde A_1u}\abs{\tilde B_1 v}.
\end{equation*}
Defining $\tau_i(x,y)_=\tilde\tau_i(A_0 x,B_0 y)$, it follows that
\begin{equation*}
  \tau(x,y)=\tilde\tau(A_0 x,B_0 y)
  =\sum_{i=0}^1\tilde\tau_i(A_0 x,B_0 y)=\sum_{i=0}^1\tau_i(x,y)
\end{equation*}
and
\begin{equation*}
  \abs{\tau_i(x,y)} =\abs{\tilde\tau_i(A_0 x,B_0 y)}
  \lesssim\begin{cases} \abs{A_0 x}\abs{B_0 y}, & i=0, \\ \abs{\tilde A_1A_0 x}\abs{\tilde B_1B_0 y}=\abs{A_1 x}\abs{B_1 y}, & i=1.\end{cases}
\end{equation*}
Thus we have the required decomposition also in this case.

\subsubsection*{Case: $A_0,B_0,A_1,B_1$ general}
Here we begin by observing that, if $(e_i)_{i=1}^m$ and $(f_j)_{j=1}^n$ are bases of $\R^m$ and $\R^n$, respectively, then
\begin{equation*}
  \tau(x,y)=\sum_{i=1}^n\sum_{j=1}^n\pair{x}{e_i}\tau(e_i,f_j)\pair{y}{f_j}\in\lspan\{\tau(e_i,f_j):1\leq i\leq m,1\leq j\leq n\}
\end{equation*}
belongs to a finite-dimensional subspace of $Z$. Replacing $Z$ by this subspace, we may assume without loss of generality that $Z$ is finite-dimensional (and in particular complete) to begin with. This will play a role in compactness argument further below.

Consider the singular value decompositions
\begin{equation*}
  A_k x=\sum_{i=1}^m \alpha_i^k \tilde e_i^k\pair{x}{e_i^k},\quad
  B_k y=\sum_{j=1}^n \beta_j^k \tilde f_j^k\pair{y}{f_j^k},
\end{equation*}
where each of the systems $(e_i^k)_{i=1}^m$, $(\tilde e_i^k)_{i=1}^m$, $(f_i^k)_{i=1}^n$, and $(\tilde f_i^k)_{i=1}^n$ is orthonormal and $\alpha_i^k,\beta_j^k\geq 0$.

For any $\eps\in(0,1)$, the matrices
\begin{equation*}
  A_\eps x:=\sum_{i=1}^m (\alpha_i^0+\eps) \tilde e_i^0\pair{x}{e_i^0},\quad
  B_\eps y=\sum_{j=1}^n (\beta_j^0+\eps) \tilde f_j^0\pair{y}{f_j^0}
\end{equation*}
are invertible. Hence the previous case guarantees the existence of bilinear forms $\tau_{0,\eps},\tau_{1,\eps}:\R^m\times\R^n\to Z$ such that
\begin{equation}\label{eq:tauEpsId}
  \tau=\tau_{0,\eps}+\tau_{1,\eps}
\end{equation}
and
\begin{equation}\label{eq:tauEpsBd}
  \Norm{\tau_{0,\eps}(x,y)}{Z}\lesssim \abs{A_\eps x}\abs{B_\eps y},\quad
  \Norm{\tau_{1,\eps}(x,y)}{Z}\lesssim \abs{A_1 x}\abs{B_1 y}.
\end{equation}
As $\eps\to 0$, it is evident that $A_\eps\to A$ and $B_\eps\to B$.
On the other hand, we have
\begin{equation*}
  \tau_{k,\eps}(x,y)
  =\sum_{i=1}^m\sum_{j=1}^n\pair{x}{e_i^k}\tau_{k,\eps}(e_i^k,f_j^k)\pair{y}{e_j^k}.
\end{equation*}
Testing the estimate \eqref{eq:tauEpsBd} with $x=e_i^k$, $y=f_j^k$ shows that
\begin{equation*}
  \Norm{\tau_{j,\eps}(e_i^k,f_j^k)}{Z}\lesssim
  \begin{cases}\abs{A_\eps e_i^0}\abs{B_\eps f_j^0}=(\alpha_i^0+\eps)(\beta_j^0+\eps), & k=0, \\
  \abs{A_1 e_i^1}\abs{B_1 f_j^1}, & k=1.\end{cases}
\end{equation*}
Hence $\tau_{k,\eps}(e_i^k,f_j^k)$ belongs to a bounded subset of $Z$. Since $Z$ is finite-dimensional (as we assumed without loss of generality in the beginning of this case), we can find a sequence $\eps_\ell\to 0$ such that each $\tau_{k,\eps_\ell}(e_i^k,f_j^k)$ converges to a limit, say $\tau_k(e_i^k,f_j^k)$ for each $1\leq i\leq m$, $1\leq j\leq n$, and $k=0,1$. We then define a bilinear forms $\tau_k:\R^m\times\R^n\to Z$ by
\begin{equation*}
   \tau_k(x,y)
  =\sum_{i=1}^m\sum_{j=1}^n\pair{x}{e_i^k}\tau_{k}(e_i^k,f_j^k)\pair{y}{e_j^k}.
\end{equation*}
Then $\tau_k(x,y)=\lim_{\ell\to\infty}\tau_{k,\eps_\ell}(x,y)$. Taking the limit of \eqref{eq:tauEpsId} along $\eps=\eps_\ell$, it follows that
\begin{equation*}
   \tau=\tau_0+\tau_1.
\end{equation*}
The same limit in \eqref{eq:tauEpsBd} gives
\begin{equation*}
  \Norm{\tau_{0}(x,y)}{Z}\lesssim \lim_{\ell\to\infty}\abs{A_{\eps_\ell} x}\abs{B_{\eps_\ell} y}=\abs{A_0 x}\abs{B_0 y},\quad
  \Norm{\tau_{1}(x,y)}{Z}\lesssim \abs{A_1 x}\abs{B_1 y}.
\end{equation*}
This completes the proof of the lemma.
\end{proof}

Corresponding to $\vec x\in X\otimes\R^n$‚ we recall the norm on $\R^n$ given by
\begin{equation*}
  \Norm{\vec u}{X(\vec x)}:=\Norm{\vec x\cdot\vec u}{X}\approx\abs{[\vec x]_X\vec u}.
\end{equation*}
Then with $\vec{\vec\alpha}\in\R^n\otimes\R^n\equiv\R^{n\times n}$, we can iterate this construction, with $\vec{\vec\alpha}$ in place of $\vec x$ and $X(\vec x)$ in place of $X$, to consider the norm on $\R^n$ given by
\begin{equation*}
  \vec u\mapsto \Norm{\vec{\vec \alpha}\cdot\vec u}{X(\vec x)}\approx \abs{[\vec{\vec\alpha}]_{X(\vec x)}\vec u}.
\end{equation*}

\begin{lemma}\label{lem:compareRedMats}
If $(\vec e_i)_{i=1}^n$ is an orthonormal basis and $\vec{\vec e}=\sum_{i=1}^n\vec e_i\otimes\vec e_i$ is the identity, then
\begin{equation*}
  [\vec{\vec e}]_{X(\vec x)}\approx [\vec x]_X.
\end{equation*}
\end{lemma}

\begin{proof}
By concatenating the various definitions, we have
\begin{equation*}
   \abs{[\vec{\vec e}]_{X(\vec x)}\vec u}
   \approx\Norm{\vec{\vec e}\cdot\vec u}{X(\vec x)}
   =\Norm{\vec u}{X(\vec x)}
   =\Norm{\vec x\cdot\vec u}{X}
   \approx\abs{[\vec x]_X\vec u}.
\end{equation*}
\end{proof}

We are now ready to prove the vector-valued self-improvement of bilinear estimates with a two-term bound:

\begin{theorem}\label{thm:2termBootstrap}
Let $X_0,X_1,Y_0,Y_1,Z$ be quasi-normed spaces, and
\begin{equation*}
  \tau:X_0\cap X_1\times Y_0\cap Y_1\to Z
\end{equation*}
a bilinear operator with estimate
\begin{equation*}
  \abs{\tau(f,g)}\leq \Norm{f}{X_0}\Norm{g}{Y_0}+\Norm{f}{X_1}\Norm{g}{Y_1}.
\end{equation*}
Then
\begin{equation*}
  \abs{\tau(\vec f,\vec g)}:=\Babs{\sum_{i=1}^n \tau(f_i,g_i)}
  \lesssim\abs{[\vec f]_{X_0}[\vec g]_{Y_0}}+\abs{[\vec f]_{X_1}[\vec g]_{Y_1}},
\end{equation*}
where the implied constant depends only on $n$.
\end{theorem}

\begin{proof}
For fixed $\vec f\in X_0\cap X_1$ and $\vec g\in Y_0\cap Y_1$, let $\sigma:\R^n\times\R^n\to Z$ be the bilinear form
\begin{equation*}
  \sigma(\vec\alpha,\vec\beta):= \sum_{i,j=1}^n \alpha_i \tau(f_i,g_j)\beta_j=\tau\Big(\sum_{i=1}^n\alpha_i f_i,\sum_{j=1}^n\beta_j g_j\Big)
  =\tau(\vec\alpha\cdot\vec f,\vec\beta\cdot\vec g).
\end{equation*}
The assumption says that
\begin{equation*}
  \Norm{\sigma(\vec\alpha,\vec\beta)}{Z}
  \leq \sum_{k=0}^1\Norm{\vec\alpha\cdot\vec f}{X_k}\Norm{\vec\beta\cdot\vec g}{Y_k}
  \approx\sum_{k=0}^1\abs{[\vec f]_{X_k}\vec\alpha}\abs{[\vec g]_{Y_k}\vec\beta}.
\end{equation*}
Then Lemma \ref{lem:bound2eq} implies the existence of bilinear forms $\sigma_k:\R^n\times\R^n\to Z$ such that
\begin{equation*}
  \sigma=\sigma_0+\sigma_1
\end{equation*}
and
\begin{equation*}
  \Norm{\sigma_k(\vec\alpha,\vec\beta)}{Z}
  \lesssim\abs{[\vec f]_{X_k}\vec\alpha}\abs{[\vec g]_{Y_k}\vec\beta}
  \approx\Norm{\vec\alpha}{X_k(\vec f)}\Norm{\vec\beta}{Y_k(\vec g)},\quad
  \forall\vec\alpha,\vec\beta\in\R^n.
\end{equation*}
By Proposition \ref{prop:normToCave}, it follows that
\begin{equation*}
  \Norm{\sigma_k(\vec{\vec\alpha},\vec{\vec\beta})}{Z}
  \lesssim\abs{[\vec{\vec\alpha}]_{X_k(\vec f)}[\vec{\vec\beta}]_{Y_k(\vec g)}},\quad \forall\vec\alpha,\vec\beta\in\R^n\otimes\R^n,
\end{equation*}
and hence
\begin{equation*}
  \Norm{\sigma(\vec{\vec\alpha},\vec{\vec\beta})}{Z}
  \lesssim\sum_{k=0}^1\abs{[\vec{\vec\alpha}]_{X_k(\vec f)}[\vec{\vec\beta}]_{Y_k(\vec g)}},\quad \forall\vec\alpha,\vec\beta\in\R^n\otimes\R^n,
\end{equation*}
In particular, choosing both $\vec{\vec\alpha}=\vec{\vec\beta}=\vec{\vec e}=\sum_{i=1}^n\vec e_i\otimes\vec e_i$ as the identity, we have
\begin{equation*}
  \sigma(\vec{\vec e},\vec{\vec e})=\tau(\vec{\vec e}\cdot\vec f,\vec{\vec e}\cdot\vec g)=\tau(\vec f,\vec g)
\end{equation*}
and, by Lemma \ref{lem:compareRedMats}
\begin{equation*}
   \abs{[\vec{\vec e}]_{X_k(\vec f)}[\vec{\vec e}]_{Y_k(\vec g)}}
   \approx\abs{[\vec f]_{X_k}[\vec g]_{Y_k}}.
\end{equation*}
Hence
\begin{equation*}
  \Norm{\tau(\vec f,\vec g)}{Z}
  =\Norm{\sigma(\vec{\vec e},\vec{\vec e})}{Z}
  \lesssim\sum_{k=0}^1 \abs{[\vec{\vec e}]_{X_k(\vec f)}[\vec{\vec e}]_{Y_k(\vec g)}}
  \approx\sum_{k=0}^1 \abs{[\vec f]_{X_k}[\vec g]_{Y_k}},
\end{equation*}
which completes the proof.
\end{proof}

Finally, we can give:

\begin{proof}[Proof of Theorem \ref{thm:mainLp}]
Let
\begin{equation*}
  X_0=X_1=F,\quad\Norm{f}{X_k}:=\Norm{A_k f}{L^{p_k}},
\end{equation*}
\begin{equation*}
 Y_0=Y_1=G,\quad\Norm{g}{Y_k}:=\Norm{B_k f}{L^{q_k}},
\end{equation*}
and $Z=L^r$. Then the assumptions of Theorem \ref{thm:2termBootstrap} are satisfied, and the said theorem gives the upper bound
\begin{equation*}
  \abs{\tau(\vec f,\vec g)}
  \lesssim\abs{[\vec f]_{X_0}[\vec g]_{Y_0}}+\abs{[\vec f]_{X_1}[\vec g]_{Y_1}}.
\end{equation*}
Moreover, each $X_k$ and $Y_k$ is a quasi-normed space related to a quasi-normed function lattice $L^{p_k}$ or $L^{q_k}$ in the sense of Definition \ref{def:lattice}. Thus Corollary \ref{cor:equiNormsQBanachFn} applies to show that
\begin{equation*}
\begin{split}
  \abs{[\vec f]_{X_k}[\vec g]_{Y_k}}
  &\approx\Norm{\vec f\odot\vec g}{X_k(Y_k)}
  =\Norm{A_k\vec f\odot B_k\vec g}{L^{p_k}(L^{q_k})} \\
  &\approx\Norm{\vec f\odot\vec g}{Y_k(X_k)}
  =\Norm{A_k\vec f\odot B_k\vec g}{L^{q_k}(L^{p_k})}.
\end{split}
\end{equation*}
\end{proof}

%\bibliography{weighted}
%\bibliographystyle{abbrv}

\end{document}